\documentclass[11pt]{article}
\usepackage{amsmath, amsthm, amsfonts}
\usepackage{latexsym}
\usepackage{graphicx, psfrag}
\makeatletter
\newfont{\footsc}{cmcsc10 at 8truept}
\newfont{\footbf}{cmbx10 at 8truept}
\newfont{\footrm}{cmr10 at 10truept}
\makeatother
\newtheorem{theorem}{\bf Theorem}
\newtheorem{proposition}{\bf Proposition}
\newtheorem{lemma}{\bf Lemma}
\newtheorem{definition}{\bf Definition}
\newtheorem{conjecture}{\bf Conjecture}
\newtheorem{corollary}{\bf Corollary}

\begin{document}
\title{On the Unique Crossing Conjecture of Diaconis and Perlman on Convolutions of Gamma Random Variables}

\author{Yaming Yu\\
\small Department of Statistics\\[-0.8ex]
\small University of California\\[-0.8ex] 
\small Irvine, CA 92697, USA\\[-0.8ex]
\small \texttt{yamingy@uci.edu}}

\date{}
\maketitle

\begin{abstract}
Diaconis and Perlman (1990) conjecture that the distribution functions of two weighted sums of iid gamma random variables cross exactly once if one weight vector majorizes the other.  We disprove this conjecture when the shape parameter of the gamma variates is $\alpha <1$ and prove it when $\alpha\geq 1$.

\end{abstract}

\section{Introduction}
Let $X_1,\ldots, X_n$ be independent and identically distributed (iid) gamma($\alpha, 1$) random variables and denote the distribution function for $\sum_{i=1}^n \theta_i X_i$ by $F_\theta$ where $\theta \equiv (\theta_1,\ldots, \theta_n)$ is a nonnegative weight vector.  Diaconis and Perlman (1990) made the following 
\begin{conjecture}
\label{conj1}
If $\eta\prec \theta$ (Marshall, Olkin and Arnold 2009), but $\theta$ is not a permutation of $\eta$, then $F_\eta(x) -F_\theta(x)$ changes signs exactly once, from $-$ to $+$, as $x$ increases from $0$ to $\infty$.  
 \end{conjecture}
 Intuitively $\sum_{i=1}^n \theta_i X_i$ is more dispersed when the weight vector $\theta$ becomes less uniform.  This conjecture, known as the unique crossing conjecture (UCC), can be seen as a strong statement about the dispersion of these weighted sums in terms of tail probabilities.  Disconis and Perlman verified the UCC in the following cases: (a) $n=2$; (b) $n=3$ and $\alpha=1$; (c) $n\geq 3,\ \alpha\geq 1$ and $\theta$ and $\eta$ differ in only two components; (d) $n\geq 3$ and components of $\eta$ are equal.  Case (d) is useful for providing bounds on $F_\theta$ in terms of the distribution function of a single gamma variable.  For this purpose Diaconis and Perlman also carried out detailed analysis of the location of the crossing point between $F_\theta$ and $F_\eta$ when all components of $\eta$ are equal.  Nevertheless, as remarked by several authors (Kochar and Xu, 2012; Roosta-Khorasani and Sz\'{e}kely, 2015) the UCC itself has remained an open problem. 

In this paper we show that the UCC holds when $\alpha \geq 1$, but may fail when $\alpha <1$, which is surprising as previous work has all supported the general validity of the UCC.  This is relevant for understanding the behavior of tail probabilities for weighted sums of $\chi^2$ random variables (corresponding to $\alpha=1/2$) which arise naturally in statistical applications.  For general $\alpha$, weighted sums of gamma variables appear in diverse areas including reliability, actuarial science and statistics, and their properties have been extensively studied; see, for example Bock et al. (1987), Diaconis and Perlman (1990), Sz\'{e}kely and Bakirov (2003), Khaledi and Kochar (2004), Zhao and Balakrishnan (2009), Yu (2009, 2011), Kochar and Xu (2012), and Roosta-Khorasani and Sz\'{e}kely (2015).

\section{Special cases and a counterexample}
Theorem~\ref{thm1} gives a positive answer to the UCC when $\alpha \geq 1$ and the weight vectors form a special configuration.  
\begin{theorem}
\label{thm1}
Suppose $\alpha \geq 1$.  Suppose $0<\theta_1\leq \cdots \leq \theta_n$ and $\eta_1\leq \cdots \leq \eta_n$ and (a) there exists $2\leq k\leq n$ such that $\theta_i<\eta_i$ for $i< k$ and $\theta_i > \eta_i$ for $i\geq k$;
(b) $\prod_{i=1}^n \eta_i > \prod_{i=1}^n \theta_i$.  Then there exists $x_0\in (0, \infty)$ such that $F_\eta(x) < F_\theta(x)$ for $x\in (0, x_0)$ and the inequality is reversed for $x> x_0$.
\end{theorem}

\begin{corollary}
\label{coro0}
The UCC is valid if $n\geq 3,\ \alpha\geq 1$ and the weight vectors differ in at most three components.
\end{corollary}
\begin{proof}
When $n=3$, conditions of Theorem~\ref{thm1} can be written as $\eta_{(1)} > \theta_{(1)},\ \eta_{(3)} < \theta_{(3)}$ and $\theta_1\theta_2\theta_3 <\eta_1\eta_2\eta_3$.  It is easy to verify that if $\eta\prec \theta$ then these conditions are satisfied (if any $\eta_i=\theta_j$ then the problem reduces to the $n=2$ case).  As noted by Diaconis and Perlman, when $\alpha\geq 1$, one may extend the validity of the UCC for $n=3$ to $n\geq 3$ provided that the weight vectors differ in at most three components.  
\end{proof}

{\bf Remark 1.} The conditions of Theorem~\ref{thm1} may be relaxed to allow $\theta_i= \eta_i$ for some $i<n$ but we cannot relax the assumption $\theta_n > \eta_n$, which ensures that $F_\eta(x) > F_\theta(x)$ for large enough $x$.  For example, letting $n=3,\, \alpha=1$, $\theta=(1, 6, 10)$ and $\eta = (4, 5, 10)$ we can directly check that $F_\theta\leq_{\rm st} F_\eta$ and there is no crossing.  When $\theta_n = \eta_n$, deciding whether there is one crossing or no crossing requires additional analysis. 

To prove Theorem~\ref{thm1} we need some preliminary results.  The following lemma is a special case of Theorem~1 of Yu (2011). 
\begin{lemma}
\label{lem1}
For $n \geq 2$ and $\alpha >0$, if $\log \eta \prec \log \theta$ then $F_\eta(x)\geq F_\theta(x)$ for all $x\in (0, \infty)$.  That is, $F_\eta$ is stochastically dominated by $F_\theta$. 
\end{lemma}

Proposition~\ref{prop1} summarizes conditions for unique crossing for $n=2$ and general $\alpha>0$.
\begin{proposition}
\label{prop1}
Suppose $n=2$.  Then $F_\eta$ crosses $F_\theta$ exactly once, and from below, as $x$ increases from $0$ to $\infty$ if and only if $\theta_1 \theta_2 < \eta_1 \eta_2$ and $\max(\theta_1, \theta_2) > \max(\eta_1, \eta_2)$.
\end{proposition}
\begin{proof}
Let us assume $\theta_1 \leq \theta_2$ and $\eta_1 \leq \eta_2$ without loss of generality. 

{\bf Necessity.} Suppose $F_\eta$ crosses $F_\theta$ exactly once from below.  Then $F_\eta(x) -F_\theta(x)$ is negative for sufficiently small $x>0$ and is positive for sufficiently large $x$.  But 
$$\frac{1 - F_\theta(x)}{1 - F_\eta(x)}  \leq \frac{\Pr(\theta_2 (X_1+X_2) > x)}{\Pr(\eta_2 X_2 > x)}.$$
The latter ratio is asymptotic to $g_{2\alpha}(x/\theta_2)/ g_\alpha(x/\eta_2)$ as $x\to \infty$ where $g_\alpha(t) \equiv  t^{\alpha -1} e^{-t} /\Gamma(\alpha)$.  It is clear that if $\theta_2 < \eta_2$ then $(1 - F_\theta(x))/(1 - F_\eta(x))\to 0$ as $x\to\infty$ and hence we must have $\theta_2> \eta_2$.  (It is easy to dismiss the boundary case $\theta_2 = \eta_2$.)  On the other hand,  
$$\lim_{x\downarrow 0} \frac{F_\theta(x)}{F_\eta(x)} = \lim_{x\downarrow 0} \frac{ f_\theta(x)
}{f_\eta(x)} = \left(\frac{\theta_1\theta_2}{\eta_1\eta_2}\right)^{-\alpha},$$
see, for example, Yu (2009), Equation (13).  Hence we must have $\theta_1\theta_2 \leq \eta_1\eta_2$.  To rule out the boundary case, note that if $\theta_1\theta_2 = \eta_1\eta_2$, then $(\log\eta_1, \log\eta_2)\prec (\log \theta_1, \log \theta_2)$ and, by Lemma~\ref{lem1}, there is no crossing.

{\bf Sufficiency.} Assume $\theta_1\theta_2 < \eta_1 \eta_2$ and $\theta_2 > \eta_2$.  Retracing the proof of the necessity part, we can deduce that $F_\eta$ crosses $F_\theta$ at least once, from below.  To show that the crossing point is unique we slightly extend the arguments of Diaconis and Perlman (1990).  We have 
$$F_\theta(x) - F_\eta(x) = x u^{-2} \int_0^\infty (H_\theta(u) - H_\eta(u)) g_{2\alpha}(x/u) {\rm d} u$$
where $H_\theta(u) = \Pr(\theta_1 W_1 + \theta_2 W_2 \leq u)$ and $W_1$ is a beta($\alpha, \alpha$) random variable with $W_1 = 1-W_2 = X_1/(X_1 + X_2).$ The kernel $g_{2\alpha}(x/u)$ is strictly totally positive (STP) for $(x, u) \in (0, \infty)^2$.  The claim would follow from variation-diminishing properties of STP kernels if we can show that $H_\theta(u) -H_\eta(u)$ changes signs only once, from $+$ to $-$, as u increases on $(0, \infty)$.  Note that 
$$H_\theta(u) - H_\eta(u) =B\left(\frac{\eta_2 - u}{\eta_2 -\eta_1}\right) - B\left(\frac{\theta_2 -u}{\theta_2 -\theta_1}\right)$$
where $B(\cdot)$ denotes the beta($\alpha, \alpha$) distribution function.  Let 
$$u^* = \frac{\theta_2 \eta_1 - \eta_2 \theta_1}{\theta_2 -\theta_1 -\eta_2 +\eta_1}.$$
Under the assumptions we have $\theta_1< \eta_1 \leq \eta_2 <\theta_2$.  It follows that $\eta_1 \leq u^* \leq \eta_2 $ and 
$$
H_\theta(u) - H_\eta(u) \begin{cases} > 0 & {\rm if}\ \theta_1 < u < u^*,\\
  < 0 & {\rm if}\  u^* < u < \theta_2,\\
  =0 & {\rm otherwise}.
 \end{cases}
$$
That is, $H_\theta(u) - H_\eta(u)$ has only one sign change at $u^*$, as needed.
 \end{proof}
{\bf Remark 2.} This Proposition is closely related to Theorem~3.6 of Kochar and Xu (2011) and Proposition~3.1 of Kochar and Xu (2012) who compare $F_\theta$ and $F_\eta$ according to the {\it star order} (Shaked and Shanthikumar, 2007).  $F_\eta \leq_{*} F_\theta$ means $F_\eta(x)$ crosses $F_\theta(cx)$ at most once, and always from below, for all $c>0$.  Proposition~\ref{prop1} can be used to recover a special case of Proposition~3.1 of Kochar and Xu (2012). 
\begin{corollary}
If $\theta_2>\theta_1,\ \eta_2>\eta_1$ and $\theta_2/\theta_1 > \eta_2/\eta_1$ then $F_\eta \leq_* F_\theta$.
\end{corollary}
 \begin{proof}
In the stated parameter configuration one can show that, for every $c>0$, either $\theta$ and $c\eta$ satisfy the necessary conditions of Proposition~\ref{prop1} and  $F_\theta$ and $F_{c\eta} $ cross exactly once, or they are ordered by the usual stochastic order, and there is no crossing.  In other words, $F_\eta \leq_* F_\theta$.
\end{proof}
\begin{proof}[Proof of Theorem~\ref{thm1}]
We use induction on $n$.  The case of $n=2$ is given by Proposition~\ref{prop1}.  Suppose $n\geq 3$.  The following argument works for $k<n$, and can be modified (with a different definition of $\tau$) to handle the $k=n$ case.  Let us consider $\tau \equiv (\theta_1, \ldots, \theta_{k-2}, \delta, \eta_{k}, \theta_{k+1}, \ldots, \theta_n),$ where  
$$\delta_*\equiv \frac{\theta_{k-1}\theta_{k}}{\eta_{k}}\leq \delta\leq \min\left(\theta_{k},\, \prod_{i\neq k} \eta_i/\prod_{i\neq k-1, k} \theta_i\right)\equiv \delta^*.$$
It is easy to see that $\delta_*<\delta^*$, and for $\delta\in (\delta_*, \delta^*)$ we have $\delta > \theta_{k-1}$ and $\prod_{i=1}^n \eta_i > \prod_{i=1}^n \tau_i > \prod_{i=1}^n \theta_i$. 
Also, $\tau_i <\eta_i$ for $i< k-1$ and $\tau_i >\eta_i$ for $i\geq k+1$ (including $i=n$ since $k<n$).  Hence the sequence $\tau_{(i)} - \eta_i,\ i=1,\ldots, n$, has exactly one sign change, whether $\delta > \eta_{k-1}$ or not.  (In the special case $k=2$ we have $\delta < \eta_1$.)  Here we use $\tau_{(i)}$ rather than $\tau_i$ to account for possible switching between $\eta_{k-1}$ and $\delta$ when we rearrange $\tau$.  As $\tau$ and $\theta$ differ in only two components, and $\tau$ and $\eta$ have at least one ($\eta_k$) in common, by the induction hypothesis, $F_\tau$ crosses $F_\theta$ at most once, from below (say at $x_*$) and $F_\eta$ crosses $F_\tau$ at most once, from below (say at $x^*$).  When $\alpha \geq 1$ the gamma density is log-concave, which ensures that adding identical components does not create multiple crossings.  It is possible that the original single crossing is annihilated.  If $\delta$ is large then $F_\tau$ could stay entirely below $F_\theta$ (effectively $x_*=\infty$).  It is not possible, however, for $F_\eta$ to stay entirely below $F_\tau$, because $\tau_n>\eta_n$. 

Note that $F_\tau$ stochastically increases in $\delta$, which implies that $x_*$ increases while $x^*$ decreases in $\delta$.  This monotonicity can then be used to show that the crossing points (as long as they are finite) are continuous functions of $\delta$.  Specifically, fix $\delta_0\in (\delta_*, \delta^*)$ such that $x_*(\delta_0)$ is finite.  Then, by the continuity of $F_\theta$ and $F_\tau$, and the monotonicity of $x_*$, we have 
$$F_\theta(x_*(\delta_0 +)) = \lim_{\delta \downarrow \delta_0}  F_\theta(x_*(\delta)) = \lim_{\delta \downarrow \delta_0} F_\tau (x_*(\delta)) = F_{\tau(\delta_0)}(x_*(\delta_0 + )).
$$
That is, $F_\theta$ and $F_{\tau(\delta_0)}$ crosses at $x_*(\delta_0 +)$.  By uniqueness of the crossing point we have $x_*(\delta_0) = x_*(\delta_0 +)$, and similarly $x_*(\delta_0) = x_*(\delta_0 -)$, showing that $x_*$ is continuous.  
 
 At $\delta = \delta_*$, we have $F_\tau\leq_{\rm st} F_\theta$ by Lemma~\ref{lem1}.  So there is no crossing between $F_\tau$ and $F_\theta$.  That is, $x_*\downarrow 0$ as $\delta\downarrow \delta_*$ and $\lim_{\delta\downarrow \delta_*} x^* >0$.  At $\delta = \theta_k$ we have $F_\tau\geq_{\rm st} F_\theta$.  Thus $x_*\uparrow \infty$ as $\delta\uparrow \theta_k$ if $\delta^*=\theta_k$.  The other possibility is 
 $\delta^* = \prod_{i\neq k} \eta_i/\prod_{i\neq k, k-1} \theta_i$.  In this case, at $\delta=\delta^*$ we have $F_\tau\geq_{\rm st} F_\eta$ again by Lemma~\ref{lem1}, because $\log\eta\prec \log\tau$.  Indeed, because $\log(\eta_i/\tau_i)$ changes signs only once (after $\tau$ is arranged in increasing order), $\sum_{i=1}^l \log(\eta_i/\tau_i)$ first increases, and then decreases.  At $l=1$ we have $\eta_1>\tau_1$.  At $l=n$ we have $\sum_{i=1}^n \log(\eta_i/ \tau_i) = 0$.  Thus $\sum_{i=1}^l \log(\eta_i/ \tau_i) \geq 0$ for all $l=1,\ldots, n$, that is, $\log\eta\prec\log\tau$. It follows that $x^*\downarrow 0$ as $\delta\uparrow \delta^*$.  
 
Regardless of whether $\delta^*=\theta_k$, we have $x_*< x^*$ as $\delta\to \delta_*$ and $x_* > x^*$ as $\delta\to \delta^*$.  By continuity, there exists some $\delta$ such that $x_*=x^*$ and 
 \begin{align*}
 F_\theta(x) &> F_\tau(x)> F_\eta(x),\quad 0<x<x^*;\\
 F_\theta(x) &< F_\tau(x)< F_\eta(x),\quad x^*<x<\infty.
 \end{align*}
It follows that $F_\eta$ crosses $F_\theta$ exactly once, from below. 
\end{proof}

Theorem~\ref{thm2} gives a negative answer to the UCC for $n=3$ and $\alpha < 1$.  Note that counterexamples for $n>3$ can be generated from a counterexample for $n=3$ by appending small enough components to the weight vectors.

\begin{theorem}
\label{thm2}
For every $0<\alpha < 1$ there exist positive vectors $\theta$ and $\eta$ with $n=3$ such that $\eta\prec \theta$ and $F_\eta(x) -F_\theta(x)$ changes signs at least three times as $x$ increases from $0$ to $\infty$. 
\end{theorem}
Theorem~\ref{thm2} is derived through a perturbation analysis rather than extensive numerical calculations.  Our counterexamples have the feature that $\theta_1 \approx \theta_2 \ll \theta_3$, and $\eta$ is a small perturbation of $\theta$ which changes all three components.  We show that for a suitable choice of such $\theta$ and $\eta$, there exists a point at which $F_\eta(x)$ crosses $F_\theta(x)$ from above.  Since $F_\eta(x)$ has to cross $F_\theta(x)$ from below for sufficiently small $x> 0$ and for sufficiently large $x$, it follows that there are at least three crossing points.  The rest of this section makes this precise.  We need the following result, which is slightly modified from Lemma~1 of Rinott et al. (2012); see also Sz\'{e}kely, G. J. and Bakirov (2003). 

\begin{lemma}
\label{lem2}
Suppose $X_i$ and $Z_i,\ i=1,2$ are independent random variables with $X_i\sim {\rm gamma}(\alpha, 1)$ and $Z_i\sim {\rm expo}(1),\ i =1, 2$, and, independently, $G$ is a weighted sum of iid gamma variates.  Fix $\theta_i^*>0,\ i=1,2$ and let $\theta_1 = \theta_1^* -\delta$ and $\theta_2 = \theta_2^* +\delta$.  Then 
$$\frac{\partial}{\partial \delta} \Pr(\theta_1 X_1 + \theta_2 X_2 + G\leq x) = \alpha (\theta_2 -\theta_1) \frac{\partial^2}{\partial x^2} \Pr(\theta_1(X_1 + Z_1) + \theta_2 (X_2 + Z_2) + G\leq x).
$$
\end{lemma}
We also need a result concerning mixtures of gamma densities ($g_\alpha(t)\equiv t^{\alpha -1} e^{-t}/\Gamma(\alpha)$).
\begin{lemma}
\label{lem3}
If $\alpha \in (0, 1)$ then there exists $\lambda \in (0, \infty)$ such that the mixture density $(\lambda g_{1+\alpha}(x) + g_{\alpha}(x))/(1+\lambda)$ is bimodal.  There exists no such $\lambda$ if $\alpha \geq 1$.
\end{lemma}
\begin{proof}
For any $x_0\in (0, \alpha)$ one may set $\lambda \equiv \lambda(x_0)= - g'_{\alpha}(x_0) / g'_{1+\alpha}(x_0)$ and obtain a function $s(x)\equiv \lambda g_{1+\alpha}(x) + g_{\alpha}(x)$ such that $s'(x_0) = 0$.  Note that since $x_0$ is between the modes of $g_{1+\alpha}$ and $g_{\alpha}$ we necessarily have $\lambda>0$.  We show that $\alpha\in (0, 1)$ is equivalent to the existence of some $x_0$ such that $s''(x_0)>0$, which indicates a local minimum.  By direct calculation we get 
$$s''(x_0)  = \lambda g''_{1+\alpha}(x_0) + g''_{\alpha}(x_0) = -\lambda'(x_0) g'_{1+\alpha}(x_0)$$
with 
$$ \lambda'(x_0) =\frac{\alpha (x_0^2 + 2(1-\alpha )x_0 - \alpha(1-\alpha))}{(\alpha x_0- x_0^2)^2}. $$
Thus $s''(x_0)$ has the opposite sign of $x_0^2 + 2(1-\alpha )x_0 - \alpha(1-\alpha)$.  If $\alpha<1$ then this quadratic is negative for sufficiently small $x_0>0$; if $\alpha \geq 1$ then it is positive for all $x_0>0$.  Thus a local minimum exists if and only if $\alpha\in (0, 1)$. 
\end{proof}

\begin{proof}[Proof of Theorem~\ref{thm2}]
Since $\alpha\in (0, 1)$, by Lemma~\ref{lem3} one can choose $\lambda>0$ such that $s(x)\equiv \lambda g_{1+\alpha}(x) +  g_{\alpha}(x)$ is bimodal, with a local minimum at $x_0>0$.  Choose $w>0$ small enough so that $s'(x_0 -w)<0$ and $s'(x_0+w)>0$. 

Let $\theta_1 = \epsilon-\delta,\, \theta_2 = \epsilon + \delta -\lambda\delta^2,\, \theta_3 = 1+\epsilon + \lambda\delta^2$, where $\epsilon$ and $\delta$ are to be determined.  We require $\epsilon>\delta>0$ and $\epsilon < 1/\lambda$ so that the weight vector $\theta$ is positive.  Let $X_i,\, i=1,2,\ldots$ be independent gamma($\alpha, 1$) variates and $Z_i\sim {\rm expo}(1)$ independently.  Define $G_0\equiv \sum_{i=1}^3 \theta_iX_i$.  Denote the density of $G_1\equiv G_0 + \sum_{i=1}^2\theta_iZ_i$ by $f_1$ and that of $G_2\equiv G_0+\sum_{i=2}^3 \theta_i Z_i$ by $f_2$.  As $\epsilon\downarrow 0$ and $\epsilon> \delta$, we have $f_1(x) \to g_{\alpha}(x)$ and $f_2(x)\to g_{1+\alpha}(x)$ pointwise in $x\in (0, \infty)$.  To show this, let $\nu= 2+3 \alpha$ and $T=\sum_{i=1}^3 X_i + \sum_{i=1}^2 Z_i$. Then $T\sim {\rm gamma}(\nu, 1)$ and we may use the independence of $T$ and $S\equiv G_1/T$ to obtain 
\begin{align}
\label{complex1}
f_1(x)  & = \frac{x^{\nu -1}}{\Gamma(\nu)} E \left[e^{-x/S} S^{-\nu}\right];\\
f'_1(x) & = \frac{\nu-1}{x} f_1(x) - \frac{x^{\nu -1}}{\Gamma(\nu)} E \left[e^{-x/S} S^{-\nu -1} \right].
\end{align}
For fixed $x>0$, the function $e^{-x/s} s^{-\nu}$ vanishes as $s\downarrow 0$ and is bounded and continuous over $s\in (0, \infty)$, achieving its maximum at $s = x/\nu$.  As $\epsilon\downarrow 0$, we have $S\to {\rm beta}(\alpha, 2 + 2\alpha)$ in distribution.  Hence $E \left[e^{-x/S} S^{-\nu}\right]$ converges, and $f_1(x) $ converges to $g_{\alpha} (x)$.  Similarly, convergence holds for the derivatives of $f_i,\ i=1,2$.  In fact, from (\ref{complex1}) we can regard $f_1(x)$ as an analytic function on the open right half of the complex plane, and the convergence just mentioned is uniform in compact subsets. 

Let $s'_\delta(x) \equiv \lambda (1 -\delta +2\lambda\delta^2) f'_2(x) + (1-\lambda\delta/2) f'_1(x)$.  By Lemma~\ref{lem2} and the chain rule, with $F_\theta(x) =\Pr(G_0\leq x)$, we have 
$$\frac{\partial}{\partial \delta}F_\theta(x) = 2\alpha\delta s'_\delta(x).$$
Since $s'_\delta(x) \to s'(x)$ as $\epsilon\downarrow 0$, we may choose $\epsilon<1/\lambda$ small enough so that, as long as $\delta <\epsilon$, we have $s'_\delta(x_0-w)<0$ and $s'_\delta(x_0+w)>0$.  Let $\eta = (\epsilon, \epsilon, 1+\epsilon)$.  Then $\eta\prec \theta$ and by the mean value theorem 
$$F_\theta(x) - F_\eta(x) = 2\alpha 
\delta\delta^*s'_{\delta^*}(x),\quad \delta^*\in (0, \delta).$$
But the right hand side is strictly negative at $x = x_0-w$ and strictly positive at $x = x_0+w$, indicating at least one sign change in $x\in (x_0 - w, x_0 + w)$.  Since $F_\theta(x) - F_\eta(x)>0$ for sufficiently small $x>0$ and $F_\theta(x) - F_\eta(x)<0$ for sufficiently large $x$, we have at least two additional sign changes, both from $+$ to $-$, in $x\in (0, x_0 - w]$ and $x\in [x_0+w, \infty)$, respectively.  
\end{proof}

\section{Main result and proof}
\begin{theorem}
\label{main.thm}
Conjecture~\ref{conj1} is valid if $\alpha \geq 1$. 
\end{theorem}
The rest of this paper is devoted to a proof of the above result.  We extend the techniques of the previous section.  By analyzing the distribution function crossing patterns of gamma convolutions we reduce the problem to a particular configuration of the weight vectors $\theta$ and $\eta$ that are sufficiently close.  For this local case, relationship between crossing points to modes of a mixture of gamma convolutions is explored.  We introduce a new stochastic order and derive monotonicity properties concerning densities of gamma convolutions.  These tools further reduce the problem, leading to Theorem~\ref{main.thm}.  

Lemma~\ref{lem.loc} shows that UCC holds locally for a particular configuration.  \begin{lemma}
\label{lem.loc}
Suppose $\alpha\geq 1$.  Given an index $1<k<n-1$ let $0<\theta_1\leq \cdots \leq \theta_{k-1}<\theta_k\leq \theta_{k+1}< \theta_{k+2}\leq\cdots \leq \theta_n$ and $\delta_i>0,\ i=1,\ldots, n$, and let $\eta$ be defined as follows.
$$
\eta_i = \begin{cases} &\theta_i + \delta_i,\quad i=1,\ldots, k-1;\\
 & \theta_i - \sum_{j=1}^{k-1} \delta_j, \quad i=k;\\
 & \theta_i + \sum_{j=k+2}^n \delta_j, \quad i=k+1;\\
 & \theta_i -\delta_i,\quad i=k+2,\ldots, n.
\end{cases}
$$
Then $F_\eta$ crosses $F_\theta$ exactly once from below if $\sum_{i\neq k, k+1} \delta_i $ is small enough. 
\end{lemma}

To deduce Theorem~\ref{main.thm} from Lemma~\ref{lem.loc}, we build on our proof of Theorems~\ref{thm1} and \ref{thm2}.  Let us introduce a majorization-type ordering which may be of independent interest.  As usual we write $\theta_{(1)},\theta_{(2)},..., \theta_{(n)}$ as $\theta$ rearranged in increasing order. 
\begin{definition} 
We say a real vector $\theta$ {\it V-majorizes} $\eta$, written as $\eta\prec_{\rm V} \theta$, if there exists $\tilde{\theta}$ such that $\eta\prec \tilde{\theta}$ and  indices $1\leq k_1, k_2\leq n$ such that 
\begin{equation}
\label{major.V}
\theta_{(i)}\begin{cases} \leq  &\tilde{\theta}_{(i)}\leq \eta_{(i)},\quad 1\leq i\leq k_1;\\
 = & \tilde{\theta}_{(i)},\quad k_1<i<k_2;\\
\geq & \tilde{\theta}_{(i)}\geq \eta_{(i)},\quad k_2\leq i \leq n.
\end{cases}
\end{equation}
\end{definition}

Simply put, $\theta$ V-majorizes $\eta$ if $\theta$ is obtained from a vector $\tilde{\theta}$ that majorizes $\eta$ by increasing (and decreasing) the largest (smallest) few components of $\tilde{\theta}$ which are already larger (smaller) than those of $\eta$.  Let us record some useful properties of $\prec_{\rm V}$.
\begin{proposition}
\label{prop.V}
Let $\eta,\, \theta$ be positive vectors such that $\eta\prec_{\rm V} \theta$.  (a) If $\prod_{i=1}^n (\eta_i/\theta_i) \geq 1$ then $\log \eta \prec^w \log\theta$.  (b) If $\prod_{i=1}^n (\eta_i/\theta_i) \leq 1$ then $F_\eta \leq_{\rm st} F_\theta$. 
\end{proposition}
\begin{proof}
Assume the coordinates of $\theta$ (respectively, $\eta$) are arranged in increasing order, and assume $\theta\neq \eta$.  Let $\tilde{\theta}$ (also arranged in increasing order) be such that $\eta\prec \tilde{\theta}$ and (\ref{major.V}) is satisfied.  

Part (a). Note that $\tilde{\theta}_i \geq \theta_i$ for $1\leq i<k_2$.  From $\eta\prec\tilde{\theta}$ we get $\log\eta\prec^w \log\tilde{\theta}$ and 
$$\prod_{i=1}^l \eta_i \geq \prod_{i=1}^l \tilde{\theta}_i \geq \prod_{i=1}^l \theta_i,\quad 1\leq l<k_2.
$$
For $i\geq k_2$ we have $\eta_i \leq \theta_i$ which means $\prod_{i=1}^l (\eta_i /\theta_i)$ between $l=k_2-1$ and $l=n$ is minimized at $l=n$.  Thus, to ensure $\log \eta \prec^w \log \theta$ we only need $\prod_{i=1}^n \eta_i \geq \prod_{i=1}^n \theta_i$. 

Part (b). We may define 
$$\theta^*_i \equiv \begin{cases} 
\theta_i & i< k_2, \\
\theta_i^{1-\lambda} \eta_i^{\lambda} & i\geq k_2;
\end{cases}
\quad \quad 
\lambda = \frac{\sum_{i=1}^n \log (\theta_i/\eta_i) }{\sum_{i=k_2}^n \log (\theta_i / \eta_i)}.$$ 
Then $\lambda\in [0, 1],\, \theta^*\leq \theta$ (coordinate-wise), and $\prod_{i=1}^n (\eta_i/\theta^*_i) = 1$. The reasoning of part (a) yields $\log\eta \prec \log \theta^*$.  By Lemma~\ref{lem1}, $F_\eta\leq_{\rm st} F_{\theta^*}\leq_{\rm st} F_\theta$. 
\end{proof}

With the notion of $\prec_{\rm V}$ we can suitably generalize Conjecture~\ref{conj1} and prove it, building on a special case, Lemma~\ref{lem.loc}, which we will establish later. 
\begin{theorem}
\label{s.main}
Suppose $\alpha\geq 1$, and $\eta,\, \theta$ are positive weight vectors such that $\eta \prec_{\rm V} \theta$, $\theta_{(n)}> \eta_{(n)}$ and $\prod_{i=1}^n (\eta_i/\theta_i) >1$.  Then there exists $x_0\in (0, \infty)$ such that $F_\eta(x) < F_\theta(x)$ for $x\in (0, x_0)$ and 
$F_\eta(x) > F_\theta(x)$ for $x>x_0$. 
\end{theorem}

\begin{proof}[Proof of Theorem~\ref{s.main}]
We will use induction on $n$.  The case of $n=3$ is covered by Theorem~\ref{thm1}.  For $n\geq 4$ assume $\theta_1\leq \cdots\leq \theta_n,\ \eta_1\leq\cdots\leq \eta_n$, and $\theta_i\neq \eta_i$ for all $i$.  Let us define
$$j \equiv \min\{i:\ \theta_i > \eta_i,\, 1\leq i\leq n\};\quad k \equiv \max\{i:\ \theta_i < \eta_i,\, 1\leq i\leq n\}.$$
These indices must exist because $\theta_n>\eta_n$ and $\theta_1<\eta_1$.  Moreover, we have $j \geq 2,\ k\leq n-1,\ j\neq k$ and $k\geq j-1$.  Consider two cases:

(a) $k=j-1$.  Then $\log(\eta_i/\theta_i)$ has only one sign change and the claim follows from Theorem~\ref{thm1}. 

(b) $k > j$.  Since $\eta \prec_{\rm V} \theta$, there exists $\tilde{\theta}$ such that $\eta\prec \tilde{\theta}$ and (\ref{major.V}) holds.  In (\ref{major.V}) we necessarily have $k_1\leq j-1$ and $k_2\geq k+1$, and hence $\tilde{\theta}_i =\theta_i,\ i=j, \ldots, k.$  Let $\delta= \eta -\tilde{\theta}$.  Then $\delta_i\geq 0$ for $i<j$ and $\delta_i \leq 0$ for $i>k$.  Define a weight vector $\tau $ parameterized by $t_1, t_2$ as follows. 
\begin{equation}
\label{tau.def}
\tau_i \equiv \begin{cases} 
\theta_i + t_1 \delta_i & i< j,\\
\theta_i - t_1 \sum_{l=1}^{j-1} \delta_l & i=j,\\
\theta_i & j<i<k;\\
\theta_i - t_2 \sum_{l=k+1}^n \delta_l & i=k,\\
\theta_i + t_2 \delta_i & i> k.
\end{cases}
\end{equation}
We require $0\leq t_1\leq c_1$ and $0\leq t_2\leq c_2$ where 
$$c_1 = \frac{\theta_j - \eta_j}{\sum_{i=1}^{j-1} \delta_i};\quad c_2 = \frac{\theta_k - \eta_k}{\sum_{i=k+1}^{n} \delta_i}.$$
Because $\eta\prec\tilde{\theta}$, we have $\sum_{i=1}^{j-1} \delta_i \geq -\delta_j > 0,\ \sum_{i=k+1}^{n} \delta_i \leq -\delta_k <0$ and $c_1, c_2\in (0, 1)$.  Define $\tilde{\tau} = \tau + \tilde{\theta} - \theta$.  Using $t_i\leq c_i,\ i=1,2$, we can show that components of $\tilde{\tau}$ are in increasing order, and that $\eta\prec \tilde{\tau}$.  Moreover, $\tau_i\leq \tilde{\tau}_i \leq \eta_i$ for $i<j$ and $\tau_i\geq \tilde{\tau}_i \geq \eta_i$ for $i>k$.  It follows that $\eta\prec_{\rm V} \tau$.  Also, from (\ref{tau.def}) and (i) $\tau_i\leq \eta_i \leq \eta_j\leq \tau_j$ for $i<j$ and (ii) $\tau_i\geq \eta_i \geq \eta_k\geq \tau_k$ for $i>k$, we can deduce that $\tau\prec \theta$,  which yields $\prod_{i=1}^n (\tau_i/\theta_i)\geq 1$.  In fact, if we have $t,\, \tilde{t}$ such that $0\leq t_i\leq \tilde{t}_i\leq c_i,\ i=1,2$, then $\eta\prec_{\rm V} \tau(\tilde{t})\prec \tau(t)$. 

Let us denote $\tau(t=(c_1, 0))$ by $\nu$, which has the feature that $\nu_j = \eta_j$.  Define 
$$\gamma_i \equiv \begin{cases} 
\theta_i^{1-c_3} \nu_i^{c_3} & i<j,\\
\nu_i & i\geq j; 
\end{cases}
\quad \quad 
c_3= \frac{\log (\theta_j/ \nu_j) }{\sum_{i=1}^{j-1} \log (\nu_i /\theta_i)}.$$ 
Then $c_3\in (0, 1)$, and $\log\gamma\prec \log\theta$.  Moreover, since $\nu_i\geq \gamma_i$ for $i<j$ and $\eta\prec_{\rm V} \nu$, we have $\eta\prec_{\rm V} \gamma$. 

Also, denote $\tilde{\nu}\equiv \tau(t=(0, c_2))$, and define $\tilde{\gamma}_i = \theta_i,\ i\neq k$ and $\tilde{\gamma}_k = \eta_k$.  We have $\tilde{\nu}_k = \eta_k$, $\tilde{\gamma}\geq \tilde{\nu},\ \tilde{\gamma}\geq \theta,$ and $\eta\prec_{\rm V} \tilde{\gamma}$. 

In view of Lemma~\ref{lem.loc}, let $\epsilon>0$ be small enough so that as long as $0<t_i\leq \epsilon,\ i=1,2$, we have $F_{\tau(t_1, t_2)}$ crosses $F_\theta$ exactly once from below.  Construct a continuum of weight vectors $\rho(s)$ such that $\rho(-3) = \gamma,\ \rho(-2) = \nu,\ \rho(-1) = \tau(\epsilon, 0),\ \rho(1) = \tau(0, \epsilon),\ \rho(2) =\tilde{\nu},\ \rho(3)= \tilde{\gamma}$, and values of $\rho(s)$ in between are defined through linear interpolation.  In particular, for $s\in (-1, 1)$ we have $\rho(s) = \tau(t_1, t_2)$ with $t_1=(1-s)\epsilon/2$ and $t_2 = (1+s)\epsilon/2$.  By the choice of $\epsilon$, we know $F_{\rho(s)}$ crosses $F_\theta$ exactly once, from below, for $s\in (-1, 1)$.  The same holds for $s\in (-3, -1]\cup [1,3)$ by Theorem~\ref{thm1}, with the possible exception of some $s$ in the upper portion of the interval $(2, 3)$.  When $s\in (2, 3)$ it is possible that $F_{\rho(s)}$ stays entirely below $F_\theta$.  This is not possible for $s\in [-2, 2]$ because $\rho(s)\prec \theta$ and the means of the two distributions are equal.  It is not possible for $s\in (-3, -2)$ because the mean of $F_{\rho(s)}$ is even smaller than that of $F_\theta$.  To verify the conditions of Theorem~\ref{thm1}, we examine the subintervals of $(-3, -1]\cup [1,3)$ and note that $\rho_{(i)}(s) - \theta_i$ has only one sign change as a function of $i=1,\ldots, n$. 

Let $x_*$ denote a crossing point between $F_\eta$ and $F_\theta$ (the conditions ensure at least one crossing).  Note that $F_{\rho(s)}$  stochastically decreases as $s\downarrow -3$, with $F_{\rho(-3)}\leq_{\rm st} F_\theta$, and $F_{\rho(s)}$ stochastically increases as $s\uparrow 3$, with $F_{\rho(3)} \geq_{\rm st} F_\theta$.  Although monotonicity need not hold when $s$ moves away from the boundary, by continuity, $F_{\rho(s)}$ must cross $F_\theta$ at precisely $x_*$ for some $s_*\in (-3, 3)$. 

Suppose $s_*\in [2, 3)$.  Then it is easy to verify $\eta\prec_{\rm V} \rho(s_*)$.  Note that $\rho(s_*)\neq \eta$ but $F_\eta(x_*) = F_{\rho(s_*)}(x_*)$.  By  Proposition~\ref{prop.V}, we have $\prod_{i=1}^n (\eta_i/\rho_i(s_*))> 1$.  Since $\eta$ and $\rho(s_*)$ have $\eta_k$ in common, by the induction hypothesis, $F_\eta$ crosses $F_{\rho(s_*)}$ exactly once, from below, at the same crossing point $x_*$ between $F_{\rho(s_*)}$ and $F_\theta$.  As in the proof of Theorem~\ref{thm1}, we conclude that $x_*$ is the only crossing point between $F_\eta$ and $F_\theta$. 

The case of $s_*\in (-3, -2]$ is similar. 

Suppose $s_*\in (-2, 2)$.  Regardless of which subinterval $s_*$ falls into, we have some $t^*\equiv(t_1^*, t_2^*)$ not identically zero such that 
\begin{equation}
\label{t.star}
F_{\tau(t^*)}(x) < F_\theta(x),\quad x\in (0, x_*);\quad F_{\tau(t^*)}(x) > F_\theta(x), \quad x> x_*.
\end{equation}
 By Proposition~\ref{prop.V}, we must have $\prod_{i=1}^n \tau_i(t^*) \leq \prod_{i=1}^n \eta_i$, that is, $t^*\in \Omega$, with
$\Omega\equiv \{t=(t_1, t_2):\ t_i \in [0, c_i],\ i=1,2;\ \prod_{i=1}^n \eta_i \geq \prod_{i=1}^n \tau_i(t)\}.$  If $t^*$ lies in the interior of $\Omega$, then repeating the entire argument with $\tau(t^*)$ in place of $\theta$ (which corresponds to $t=(0, 0)$) we conclude that, either the claim does hold, or there exists $t_i^{**}\geq t_i^*$, with strict inequality for at least one $i=1,2$, such that $F_{\tau(t^{**})}$ crosses $F_{\tau(t^{*})}$ (and hence $F_\theta$) exactly once, from below, at $x_*$.  And $t^{**}\in \Omega$.  Let $\Omega_0$ be the set of $t\in \Omega$ such that (i) $t\geq t^*$, and (ii) $F_{\tau(t)}(x)\leq F_{\tau(t^*)}(x)$ for $x\in (0, x_*)$ and the inequality is reversed for $x> x_*$.  By continuity, $\Omega_0$ is a closed set.  Let $\omega^*$ be an element of $\Omega_0$ with maximal value of $\omega_1 + \omega_2$.  The above discussion shows that, either the claim holds, or $\omega^*$ does not belong to the interior of $\Omega$, that is, $ \omega^*_i=c_i$ for at least one $i=1,2$.  We can rule out the other boundary situation $\prod_{i=1}^n \eta_i =\prod_{i=1}^n \tau_i(\omega^*)$ in view of Proposition~\ref{prop.V}, unless $\tau(\omega^*)$ is a permutation of $\eta$, in which case the claim follows from the definition of $\Omega_0$ and the strict inequalities (\ref{t.star}).  In other cases, by the induction hypothesis, $F_\eta$ crosses $F_{\tau(\omega^*)}$ exactly once from below, at $x_*$; the claim follows from this, the definition of $\Omega_0$, and (\ref{t.star}). 
\end{proof}

To treat the local case of Lemma~\ref{lem.loc}, a key tool is the following Lemma~\ref{analytic}, which connects whether there are multiple crossing points when the weight vector is perturbed locally to whether mixtures of several gamma convolutions are always unimodal. 

\begin{lemma}
\label{analytic}
For a fixed positive weight vector $\theta$, let $\eta$ be defined by $\eta = \theta +\sum_{k=1}^K \tau^{(k)}$ where, associated with each $k$, we have a pair of indices $i_k\neq j_k$ and a real number $\delta_k>0$ such that $\theta_{i_k}< \theta_{j_k}$ and 
$$\tau^{(k)}_i = \begin{cases} \delta_k & i=i_k,\\
  -\delta_k & i=j_k,\\ 
 0 & {\rm otherwise}.
  \end{cases}
$$ 
Let $f_k(x|\delta),\ \delta\equiv (\delta_1, \ldots, \delta_K),$ denote the density of $\sum_{i=1}^n \eta_i X_i + \eta_{i_k} Z_{i_k} + \eta_{j_k} Z_{j_k}$, where $X_i\sim {\rm gamma}(\alpha, 1),\, Z_i\sim {\rm expo}(1)$ are mutually independent.  Suppose, for arbitrary constants $\lambda_k\geq 0$ such that $\sum_k \lambda_k =1$, we have $\sum_k \lambda_k f_k(x|0)$ is unimodal, with a strictly negative second derivative at the mode, and no saddle points.  Then for small enough $\sum_{k=1}^K \delta_k$, $F_\eta$ crosses $F_\theta$ exactly once, from below. 
\end{lemma}
\begin{proof}
Note that $\eta\prec\theta$ if $\delta$ is small enough, and hence $F_\eta$ crosses $F_\theta$ at least once, from below.  By Lemma~\ref{lem2} we have 
$$\frac{\partial{F_{\eta}(x)}}{\partial \delta_k}  = \alpha (\eta_{i_k} - \eta_{j_k})  f'_k(x| \delta),\quad k=1,\ldots, K.$$
Then, 
\begin{equation}
\label{eqn.inte}
F_{\eta}(x) - F_\theta(x) = \int_0^1 \sum_{k=1}^K \alpha (\theta_{i_k} - \theta_{j_k}+ 2t\delta_k) \delta_k  f'_k(x|t\delta)\, {\rm d} t.
\end{equation}
Suppose the claim does not hold, and there exists a sequence $\eta(l)$ corresponding to $\delta(l)\equiv (\delta_{1l}, \ldots, \delta_{Kl})$ such that $\delta(l) \to 0$ and $F_{\eta(l)}(x)-F_\theta(x)=0$ has at least two roots in $x\in (0,\infty)$, for each $l=1,2,\ldots$.  Denote $L(\delta) = \sum_{k=1}^K \delta_k (\theta_{i_k} - \theta_{j_k})$. 
 By taking subsequences if necessary as $l\to\infty$, we may assume $\delta_{kl}(\theta_{i_k} -\theta_{j_k})/L(\delta(l)) \to \lambda_k$ for some nonnegative $\lambda_k$ such that $\sum_{k=1}^K \lambda_k = 1$.  We benefit from the fact that $f_k$ can be regarded as analytic functions on the open right half of the complex plane, and as $\delta$ tends to zero they converge uniformly on compact subsets.  It follows that 
\begin{equation}
\label{eqn.cross}
\lim_{l\to\infty} \frac{F_{\eta(l)}(x) - F_\theta(x)}{ \alpha L(\delta(l))} = \sum_{k=1}^K \lambda_k f_k'(x|0),
\end{equation}
and the convergence is uniform on compact subsets.  Let $D$ denote a finite horizontal strip within the open right half plane such that $D\cap \mathbb{R}$ encloses all possible roots $x\in (0, \infty)$ of $F_{\eta(l)}(x) - F_\theta(x)=0$ for sufficiently large $l$.  This is possible from bounds on location of the crossing points (Bock et al. 1987; Roosta-Khorasani and Sz\'{e}kely, 2015).  In fact, when $\delta$ is sufficiently small and $t\in (0,1)$, letting $M=\max(\theta)$ and $m=\min(\theta)$ we obtain ${\rm gamma}(n\alpha+2, m)\leq_{\rm lr} f_k(\cdot|t\delta)\leq_{\rm lr} {\rm gamma}(n\alpha +2, M)$, which implies $f_k'(x|t\delta) <0$ for $x> M/(n\alpha +1)$ and $f_k'(x|t\delta)>0$ for $x<m/(n\alpha+1)$.  It follows from (\ref{eqn.inte}) 
 that all positive real roots of $F_{\eta}(x) - F_\theta(x)=0$ must be between $m/(n\alpha+1)$ and $M/(n\alpha+1)$.  We can make $D$ thin enough so that there are no other roots of $\sum_{k=1}^K \lambda_k f_k'(x|0)$ within $D$ except for the unique mode of the real function $\sum_{k=1}^K \lambda_k f_k(x|0)$ which, by assumption, must be a simple root.  By (\ref{eqn.cross}), for large enough $l$, the number of roots of $F_{\eta(l)} - F_\theta$ within $D$, counting multiplicity, must be equal to one, which contradicts the assumption of multiple real roots.
\end{proof}

Lemma~\ref{lem.loc} is a consequence of Lemma~\ref{uni.modal} and Lemma~\ref{analytic}. 
\begin{lemma}
\label{uni.modal}
In the setting of Lemma~\ref{lem.loc}, let $X_i\sim {\rm gamma}(\alpha)$ and $Z_i\sim {\rm expo}(1)$ be mutually independent.  Let $f_j$ denote the density of $\sum_{i=1}^n \theta_i X_i + \theta_j Z_j +\theta_k Z_k$ for $j=1,\ldots, k-1$ and that of
$\sum_{i=1}^n \theta_i X_i + \theta_j Z_j +\theta_{k+1} Z_{k+1}$ for $j=k+2, \ldots, n$. 
Then, for any $\lambda =(\lambda_1, \ldots, \lambda_{k-1}, \lambda_{k+2},\ldots, \lambda_n)$ such that $\lambda_i\geq 0$ and $\sum \lambda_i =1$, the mixture density $\sum_{i\neq k, k+1} \lambda_i f_i $ is unimodal with a strictly negative second derivative at the mode, and no saddle points. 
\end{lemma}

Lemma~\ref{uni.modal} requires detailed analysis.  As a starting point, we prove some monotonicity properties concerning the densities of gamma convolutions in a simple case.  

\begin{lemma}
\label{lem.initial}
For $\theta\in (0, 1)$, let $X\sim {\rm gamma}(\alpha, 1)$ and $Z\sim {\rm expo}(1)$ independently.  Denote the density of $X + \theta Z$ by $h(x)$.  Then (a) if $\alpha \geq 1$ then $h'(x)/g_\alpha'(x)$ strictly increases in $x\in (\alpha -1,\, \infty)$; (b) if $\alpha \geq 2$ then $h'(x)/g'_\alpha(x)$ also strictly increases in $x\in (0, \alpha -1)$;  
 (c) if $\alpha \geq 1$ then $h'(x)/g_{\alpha +1}'(x)$ strictly decreases in each of $x\in (0, \alpha)$ and $x\in (\alpha, \infty)$; (d) parts (a)--(c) still hold when the distribution of $Z$ is replaced by a mixture of exponentials with rates $> 1$.
\end{lemma}
\begin{proof}
In the $\alpha=1$ case the densities are amenable to direct calculations.  Let us assume $\alpha>1$.  Denote $g\equiv g_\alpha$.  We have 
\begin{equation}
\label{eqn.lap}
\theta h'(x) + h(x) = g(x),\quad x>0,
\end{equation}
which can be verified by comparing the Laplace transform of both sides.  To prove (a), we will show $h''(x) g'(x) > h'(x) g''(x) $ for $x> \alpha -1$.  Differentiating (\ref{eqn.lap}) to eliminate $h''(x)$, and noting that $\theta g''(x) +g'(x)<0$ for $x>\alpha -1$, we equivalently need to show 
\begin{equation}
\label{eqn.u1}
u(x)\equiv e^{x/\theta}  \left[\frac{g'^2(x)}{\theta g''(x) +g'(x)}-h'(x)\right] <0,\quad x>\alpha-1.
\end{equation}
This holds for $x=\alpha -1$ because $g'(\alpha -1)=0$ and $h$ dominates $g$ in the likelihood ratio order.  By direct calculation, we have 
\begin{equation}
\label{eqn.u2}
u'(x) = \frac{\theta e^{x/\theta} g'(x)(g''^2(x) - g'(x) g'''(x))}{(\theta g''(x) +g'(x))^2} <0,\quad x>\alpha -1.
\end{equation}
Thus $u(x)<0$ for all $x>\alpha -1$, as required.

To prove (b) we only need to show that $u(x)>0$ for $x\in (0, x_*)$ and $u(x)<0$ for $x\in (x_*, \alpha -1)$ where $x_*<\alpha -1 $ is the unique positive root of $\theta g''(x) +g'(x)=0$.  When $\alpha \geq 2$ we have $u(x)\to 0$ as $x\downarrow 0$ and $u'(x)>0$ for $x\in (0, x_*)\cup (x_*, \alpha -1)$. We obtain the desired sign pattern of $u$ as a consequence.

To prove (c) we similarly will show $h''(x)g'_{\alpha+1}(x) < h'(x) g''_{\alpha+1}(x)$ for $x>0$.  This is equivalent to $\tilde{u}(x)<0$ for $x\in (0, x^*)$ and $\tilde{u}(x)>0$ for $x\in (x^*, \infty)$, where $x^*\in (\alpha -1, \alpha)$ is the unique positive root of $\theta g''_{\alpha+1}(x) +g'_{\alpha+1}(x)=0$ and 
$$\tilde{u}(x)\equiv e^{x/\theta}  \left[\frac{g'(x)g'_{\alpha+1}(x)}{\theta g''_{\alpha+1}(x) +g'_{\alpha+1}(x)}-h'(x)\right].$$
This sign pattern can be proved by arguments parallel to the previous parts.

Part (d) is obvious.
 \end{proof}

The usefulness of these monotonicity properties is more apparent after we define the following stochastic order.

\begin{definition}
Suppose $f$ and $g$ are twice continuously differentiable densities supported on an interval $I\subset (0, \infty)$.  We say $f$ is dominated by $g$ in the {\it supplemented likelihood ratio} ordering, written as $f\leq_{\rm slr} g$, if (a) $f'(x)g(x)\leq f(x)g'(x)$ for all $x\in I$ and (b) $f'(x)/ g'(x)$ decreases in each of the sets $\{x:\ f'(x) >0\}$ and $\{x:\ g'(x) <0\}$.  
\end{definition}

Some properties of $\leq_{\rm slr}$ are summarized as follows.
\begin{proposition}
\label{slr.simple}
(a) If $f\leq_{\rm slr}g $ and $g\leq_{\rm slr} h$ then $f\leq_{\rm slr} h$.  (b) If $\alpha \geq 2$ 
and $\theta\in (0, 1)$, then $g_\alpha \leq_{\rm slr} g_\alpha * {\rm expo}(\theta)$ where $*$ denotes convolution. (c) Suppose $\alpha\geq 1$, and $\theta\in (0, 1)$, then $g_\alpha * {\rm expo}(\theta)\leq_{\rm slr} g_{\alpha+1}$. (d) Parts (b) and (c)  still hold if ${\rm expo}(\theta)$ is replaced by a mixture of exponentials with rates $\geq 1$. 
\end{proposition}
\begin{proof}
Part (a) is obvious.  Parts (b)--(d) are restating Lemma~\ref{lem.initial}.  
\end{proof}
\begin{proposition}
\label{slr.prop}
Suppose $f$ and $g$ are unimodal (see Remark 3), $f\leq_{\rm slr} g$ and $h$ is Polya frequency order 3.  Assume we can take the derivatives inside the absolutely convergent integrals and obtain $(f* h)' = f' * h$ and $(g* h)' = g' * h$. Then $f*h \leq_{\rm slr} g*h$.
\end{proposition}
\begin{proof}
Let us denote $\tilde{f} = f*h$ and $\tilde{g} = g*h$.  Since $h$ is PF3, the likelihood ratio ordering is preserved, that is, $\tilde{f}\leq_{\rm lr} \tilde{g}$.  Moreover, $\tilde{f}$ and $\tilde{g}$ are unimodal.  Let $\lambda >0$ and consider the function $f'(x)-\lambda g'(x)$.  By assumption $f'(x)/g'(x)$ decreases on each of $I_+ \equiv \{x:\ f'(x)>0\}$ and $I_-\equiv \{x:\ g'(x)<0\}$.  Assume these are non-empty, otherwise the argument can be suitably modified.  Note that by $f\leq_{\rm lr} g$ we have $g'(x)>0$ for $x\in I_+$ and $f'(x)<0$ for $x\in I_-$.  On the set $I_0\equiv \{x:\ g'(x)\geq 0,\ f'(x)\leq 0\}$ we have $f'(x) -\lambda g'(x)\leq 0$.  Overall $f'-\lambda g'$ changes signs at most twice, and the sign sequence is $+, -, +$ in the case of two changes.  By the variation-diminishing properties of totally positive kernels (Karlin 1968), the same is true for $(f' -\lambda g')* h = \tilde{f}' -\lambda \tilde{g}'$.  We need to show $\tilde{f}'/\tilde{g}'$ decreases on each of $\tilde{I}_+\equiv \{x:\ \tilde{f}'(x)>0\}$ and $\tilde{I}_-\equiv \{x:\ \tilde{g}'(x)<0\}$.  Denote the upper end point of $\tilde{I}_+$ by $x_0$.  For $0<\lambda < \tilde{f}(x_0)/\tilde{g}(x_0)$, if $\tilde{f}'(x)-\lambda \tilde{g}'(x)$ ever crosses zero from below in $x\in I_+$, then it must be nonnegative for $x\geq x_0$, in order not to violate the sign pattern of $+, -, +$.
Thus 
$$\tilde{f}(x_0) = -\int_{x_0}^\infty \tilde{f}'(x)\, {\rm d}x \leq -\lambda \int_{x_0}^\infty \tilde{g}'(x)\, {\rm d}x = \lambda \tilde{g}(x_0)$$
which contradicts $\lambda < \tilde{f}(x_0)/\tilde{g}(x_0)$.  With a small perturbation this still applies when $\lambda = \tilde{f}(x_0)/\tilde{g}(x_0)$.  For $\lambda > \tilde{f}(x_0)/\tilde{g}(x_0)$, we have $\tilde{f}'(x)/ \tilde{g}'(x) < \lambda $ for $x\in I_+$ sufficiently close to $x_0$, because of the likelihood ratio ordering.  In order not to violate the sign pattern, $\tilde{f}'(x)-\lambda \tilde{g}'(x)$ cannot cross zero from below in $x\in I_+$ in this case either.  Because $\lambda$ is arbitrary, $\tilde{f}'(x)/\tilde{g}'(x)$ must decrease for $x\in I_+$.  The case of $x\in \tilde{I}_-$ is similar. 
\end{proof}

{\bf Remark 3.} We impose a restricted form of unimodality, which is satisfied by the gamma convolutions.  For the above proof to be valid, we need the set $I_0$ to be situated between $I_+$ and $I_-$.  This will be satisfied if we assume the closures of $I_+$ and $I_-$ are intervals.  So, an isolated saddle point is allowed, but not a flat ridge.  We will note down such restrictions when needed.

Proposition~\ref{comp.prop} reveals the intimate relation between $\leq_{\rm slr}$ and the unimodality of the mixture of two densities with arbitrary mixing proportions.  It allows us to reduce the problem of unimodality needed in Lemma~\ref{uni.modal} to manageable special cases.

\begin{proposition}
\label{comp.prop}
Let $f_i$ and $h_i,\ i=1,2,$ be twice continuously differentiable and unimodal densities supported on $(0, \infty)$ such that 
$$h_1\leq_{\rm slr} f_1\leq_{\rm lr} f_2\leq_{\rm slr} h_2.$$  
Suppose the mixture density $p h_1 + (1- p) h_2$ is unimodal for all $p \in [0,1]$.  Then so is $p f_1 + (1-p) f_2$, assuming $f'_1$ and $f_2'$ do not vanish simultaneously in between the modes of $f_1$ and $f_2$. 
\end{proposition}
\begin{proof}
Let $x_*$ and $x^*$ denote the modes of $f_1$ and $f_2$ respectively.  In the case of a possible plateau, $x_*$ (respectively, $x^*$) denotes the leftmost (respectively, rightmost) mode of $f_1$ (respectively, $f_2$).  Obviously all modes of the mixture $p f_1 + (1-p) f_2 $ are in the interval $[x_*, x^*]$.  Moreover, for each $x_0\in (x_*, x^*)$ such that $f_i'(x_0)\neq 0,\ i=1,2$, we may set $\lambda = - f_1'(x_0)/ f_2'(x_0)$ to obtain a stationary point of this mixture density.  By the likelihood ratio ordering, we necessarily have $f_1'(x_0)<0<f_2'(x_0)$ and $\lambda >0$.  To show that the mixture is unimodal, suppose $f_2'$ does not vanish on $(x_*, x^*)$.  Then we can show that $-f_1'(x)/f_2'(x)$ increases on $(x_*, x^*)$, which is equivalent to 
\begin{equation}
\label{mix.cond}
f_1''(x_0) f_2'(x_0) \leq f_2''(x_0) f_1'(x_0),\quad x_0\in (x_*, x^*).
\end{equation}
Condition (\ref{mix.cond}) is necessary because, if the mixture is unimodal, then a stationary point can never be a local minimum, and hence $f_1''(x_0) +\lambda f_2''(x_0) \leq 0$.  On the other hand, if $\lambda = - f_1'(x_0)/f_2'(x_0)$ is an increasing function of $x_0\in (x_*, x^*)$, then stationary points of the mixture corresponding to the same $\lambda$ form a connected interval, showing that the mixture is unimodal.  A close inspection shows that (\ref{mix.cond}) is sufficient as long as the saddle points of $f_1$ and $f_2$ on $(x_*, x^*)$ do not coincide. 

Applying this criterion to the mixture $ph_1 +(1-p) h_2 $ we have
$$h_1''(x_0) h_2'(x_0) \leq h_2''(x_0) h_1'(x_0),\quad x_0\in (x_*, x^*).$$
which yields, as long as $f_1'(x_0)\neq 0\neq f_2'(x_0)$, 
$$ \frac{f_2''(x_0)}{f_2'(x_0)}\leq \frac{h_2''(x_0)}{h_2'(x_0)}\leq \frac{h_1''(x_0)}{h_1'(x_0)} \leq \frac{f_1''(x_0)}{f_1'(x_0)},\quad x_0\in (x_*, x^*),$$
in view of  $h_1\leq_{\rm slr} f_1$ and $f_2\leq_{\rm slr} h_2$, and (\ref{mix.cond}) is established. 
\end{proof}
Next, we present two log-concavity results needed in the proof of Lemma~\ref{uni.modal}.
\begin{lemma}
\label{lem.lc1}
Suppose $X_i\sim {\rm gamma}(\alpha)$ and $Z_i\sim {\rm expo}(1),\ i=1,2$, are mutually independent where $\alpha\geq 1$.  Let $\delta_1, \delta_2 >0$.  Then arbitrary mixtures of $\delta_1 X_1 + \delta_2 X_2$ and $\delta_1 (X_1 + Z_1) + \delta_2(X_2 + Z_2)$ are unimodal. 
\end{lemma}
\begin{proof}
We show that when $\alpha=1$, such mixtures are log-concave.  If $\alpha >1$ then we can write $\sum_{i=1}^2 \delta_i X_i = \sum_{i=1}^2 \delta_i (X_i^*  + Y_i) $ where $X_i^*\sim {\rm expo}(1)$ and $Y_i \sim {\rm gamma}(\alpha -1)$ independently.  We can similarly ``split off'' $\delta_1 Y_1 + \delta_2 Y_2$ from $\sum_{i=1}^2 \delta_i (X_i + Z_i)$.  Because $\delta_1 Y_1 + \delta_2 Y_2$ is unimodal, the result follows from the log-concave result in the $\alpha=1$ case. 

Let us assume $\delta_2=1$ and $\delta\equiv \delta_1\in (0,1)$.  When $\alpha=1$, the densities of $\delta X_1 + X_2$ and $\delta (X_1 +Z_1) + X_2 +Z_2$ are, respectively, 
$$h_1(x) = \frac{e^{-x} - e^{-x/\delta }}{1-\delta}; \quad h_2(x) = \frac{x(e^{-x}+ e^{-x/\delta})-2\delta h_1(x)}{(1-\delta)^2}.$$
For $\lambda > -2\delta$ and $\epsilon \equiv \delta^{-1} -1$ let 
$$q(x)\equiv x(e^{\epsilon x} + 1) +\lambda (e^{\epsilon x}  - 1).$$  
We only need to show that $q(x)$ is log-concave.  A quick calculation yields 
$$e^{-\epsilon x}\left[q'^2(x)  - q''(x) q(x)\right] = e^{\epsilon x} + e^{-\epsilon x} -2 - (\epsilon x)^2 + (\lambda \epsilon + 2)^2 $$
which is positive for all $x>0$.
\end{proof}
\begin{lemma}
\label{lem.lc2}
Let $Y$ be an arbitrary mixture of $k\geq 1$ exponentials with means $\delta_i,\ i=1,\ldots, k$ such that $\max(\delta_i)\leq \delta$.  Let $Z\sim {\rm expo}(\delta)$ independently of $Y$.  Then $Y+Z$ is strictly log-concave. 
\end{lemma}
\begin{proof}
Suppose $\max(\delta_i)<\delta$.  The density of $U\equiv Y+Z$ can be written as 
$h(u) = \sum_{i=1}^k \lambda_i (e^{-u/\delta} - e^{-u/\delta_i}) $ for some constants $\lambda_i>0$.  We know $h(u)$ is strictly log-concave on $(0, \infty)$ because $e^{u/\delta} h(u)$ is strictly concave.  A slight modification works when $\max(\delta_i) =\delta$. 
\end{proof}

We are now ready to present the proof of Lemma~\ref{uni.modal}, which concludes the derivation of our main result. 
\begin{proof}[Proof of Lemma~\ref{uni.modal}]
We shall use the notation $\leq_{\rm slr}$ with the random variables as well as the densities.  Let $Y_1$ be an arbitrary mixture of $\theta_j Z_j$ for $j=1,\ldots, k-1$; let $Y_2$ be an arbitrary mixture of $\theta_j Z_j$ for $j=k+2, \ldots, n$.  Then $Y_1\leq_{\rm lr} \theta_k Z_k\leq_{\rm lr} Y_2$, and these have strictly decreasing densities.  Define 
$$\tilde{X}\equiv \theta_n X_n + \theta_{k+1} X_{k+1},\quad W_1 \equiv \tilde{X} + \theta_k Z_k + Y_1,\quad W_2 \equiv \tilde{X} +\theta_{k+1} Z_{k+1} + Y_2.$$  
We have $W_1\leq_{\rm lr} W_2$, and $W_1, W_2$ are unimodal (since $\alpha\geq 1$); $W_1$ is in fact log-concave by Lemma~\ref{lem.lc2}.  If we can show that arbitrary mixtures of $W_1$ and $W_2$ are unimodal, then so are those of $f_j,\ j=1,\ldots, k-1, k+2,\ldots, n,$ by adding $\sum_{i\neq k+1, n} \theta_i X_i$, which is log-concave.  

Lemma~\ref{lem.initial} yields $\theta_n X_n + Y_2 \leq_{\rm slr} \theta_n (X_n + Z_n)$. Convolving both sides with $\theta_{k+1} (X_{k+1} + Z_{k+1})$, which is PF3 (Karlin 1968), we obtain 
$$W_2\leq_{\rm slr} \theta_{k+1}(X_{k+1}+Z_{k+1})+\theta_n (X_n +Z_n).$$
Lemma~\ref{lem.initial} also yields
\begin{equation}
\label{slr.ineqs}
\theta_n(X_n+Z_n) \leq_{\rm slr} \theta_n(X_n+Z_n) + \theta_k Z_k  \leq_{\rm slr} 
 \theta_n(X_n+Z_n) + \theta_k Z_k + Y_1,
 \end{equation}
where the second $\leq_{\rm slr}$ is obtained by convolving $\theta_n(X_n+Z_n) \leq_{\rm slr} \theta_n(X_n+Z_n) + Y_1$ with $\theta_k Z_k$.  Convolving (\ref{slr.ineqs}) with $\theta_{k+1}(X_{k+1}+Z_{k+1})$ yields
$$W_2\leq_{\rm slr}  W_3\equiv \theta_{k+1}(X_{k+1}+Z_{k+1})+\theta_n (X_n +Z_n) + \theta_k Z_k + Y_1.$$
By Proposition ~\ref{comp.prop}, we only need to show that arbitrary mixtures of $W_1$ and $W_3$ are unimodal.  But this is a consequence of Lemma~\ref{lem.lc1}, which shows that arbitrary mixtures of $\tilde{X}$ and $ \theta_{k+1} (X_{k+1}+Z_{k+1}) + \theta_n (X_n + Z_n)$ are unimodal, and Lemma~\ref{lem.lc2}, which shows that $\theta_k Z_k + Y_1$ is log-concave. 


Strict unimodality, in the sense of a strictly negative second derivative at the mode, and no saddle points, can be established by a careful examination of the above steps.  For example, in Proposition~\ref{comp.prop}, the claim still holds if we use strict unimodality in both the condition on $h_i$ and the conclusion on $f_i$.  Also, in addition to being unimodal, the density of a mixture of $\tilde{X}$ and $ \theta_{k+1} (X_{k+1}+Z_{k+1}) + \theta_n (X_n + Z_n)$ is analytic on $(0, \infty)$, vanishes at $0+$, and has a bounded first derivative.  One can then verify that the step of adding $\theta_k Z_k +Y_1$, which is strictly log-concave, yields a strictly unimodal density. 
\end{proof}

\end{document}